\documentclass[reqno, 12pt]{amsart}
\pdfoutput=1
\makeatletter
\let\origsection=\section \def\section{\@ifstar{\origsection*}{\mysection}}
\def\mysection{\@startsection{section}{1}\z@{.7\linespacing\@plus\linespacing}{.5\linespacing}{\normalfont\scshape\centering\S}}
\makeatother

\usepackage{amsmath,amssymb,amsthm}
\usepackage{mathrsfs}
\usepackage{mathabx}\changenotsign
\usepackage{dsfont}

\usepackage[dvipsnames]{xcolor}
\usepackage[backref]{hyperref}
\hypersetup{
    colorlinks,
    linkcolor={red!60!black},
    citecolor={green!60!black},
    urlcolor={blue!60!black}
}

\usepackage{graphicx}

\usepackage[open,openlevel=2,atend]{bookmark}

\usepackage[abbrev,msc-links,backrefs]{amsrefs}
\usepackage{doi}

\renewcommand{\PrintDOI}[1]{\doi{#1}}

\usepackage[T1]{fontenc}
\usepackage{lmodern}
\usepackage[babel]{microtype}
\usepackage[english]{babel}

\usepackage{geometry}
\numberwithin{equation}{section}
\numberwithin{figure}{section}

\usepackage{enumitem}
\def\rmlabel{\upshape({\itshape \roman*\,})}

\usepackage{pgf,tikz}

\makeatletter
\def\greek#1{\expandafter\@greek\csname c@#1\endcsname}
\def\Greek#1{\expandafter\@Greek\csname c@#1\endcsname}
\def\@greek#1{\ifcase#1
	\or $\alpha$%
	\or $\beta$%
	\or $\gamma$%
	\or $\delta$%
	\or $\epsilon$%
	\or $\zeta$%
	\or $\eta$%
	\or $\theta$%
	\or $\iota$%
	\or $\kappa$%
	\or $\lambda$%
	\or $\mu$%
	\or $\nu$%
	\or $\xi$%
	\or $o$%
	\or $\pi$%
	\or $\rho$%
	\or $\sigma$%
	\or $\tau$%
	\or $\upsilon$%
	\or $\phi$%
	\or $\chi$%
	\or $\psi$%
	\or $\omega$%
\fi}
\def\@Greek#1{\ifcase#1
	\or $\mathrm{A}$%
	\or $\mathrm{B}$%
	\or $\Gamma$%
	\or $\Delta$%
	\or $\mathrm{E}$%
	\or $\mathrm{Z}$%
	\or $\mathrm{H}$%
	\or $\Theta$%
	\or $\mathrm{I}$%
	\or $\mathrm{K}$%
	\or $\Lambda$%
	\or $\mathrm{M}$%
	\or $\mathrm{N}$%
	\or $\Xi$%
	\or $\mathrm{O}$%
	\or $\Pi$%
	\or $\mathrm{P}$%
	\or $\Sigma$%
	\or $\mathrm{T}$%
	\or $\mathrm{Y}$%
	\or $\Phi$%
	\or $\mathrm{X}$%
	\or $\Psi$%
	\or $\Omega$%
\fi}

\AddEnumerateCounter{\greek}{\@greek}{24}
\AddEnumerateCounter{\Greek}{\@Greek}{12}

\makeatother

\let\polishlcross=\l
\def\l{\ifmmode\ell\else\polishlcross\fi}

\def\paragraph#1{%
  \noindent\textbf{#1.}\enspace}

\let\emptyset=\varnothing
\let\setminus=\smallsetminus

\makeatletter
\def\moverlay{\mathpalette\mov@rlay}
\def\mov@rlay#1#2{\leavevmode\vtop{   \baselineskip\z@skip \lineskiplimit-\maxdimen
   \ialign{\hfil$\m@th#1##$\hfil\cr#2\crcr}}}
\newcommand{\charfusion}[3][\mathord]{
    #1{\ifx#1\mathop\vphantom{#2}\fi
        \mathpalette\mov@rlay{#2\cr#3}
      }
    \ifx#1\mathop\expandafter\displaylimits\fi}
\makeatother

\DeclareFontFamily{U}  {MnSymbolC}{}
\DeclareSymbolFont{MnSyC}         {U}  {MnSymbolC}{m}{n}
\DeclareFontShape{U}{MnSymbolC}{m}{n}{
    <-6>  MnSymbolC5
   <6-7>  MnSymbolC6
   <7-8>  MnSymbolC7
   <8-9>  MnSymbolC8
   <9-10> MnSymbolC9
  <10-12> MnSymbolC10
  <12->   MnSymbolC12}{}
\DeclareMathSymbol{\powerset}{\mathord}{MnSyC}{180}

\usepackage{tikz}
\usetikzlibrary{calc,decorations.pathmorphing}
\pgfdeclarelayer{background}
\pgfdeclarelayer{foreground}
\pgfdeclarelayer{front}
\pgfsetlayers{background,main,foreground,front}

\let\epsilon=\varepsilon
\let\eps=\epsilon
\let\rho=\varrho
\let\theta=\vartheta
\let\kappa=\varkappa

\def\EE{{\mathds E}}
\let\E=\EE

\def\PP{{\mathds P}}
\let\Prob=\PP

\newcommand{\cA}{\mathcal{A}}

\newcommand{\cE}{\mathcal{E}}

\theoremstyle{plain}
\newtheorem{thm}{Theorem}[section]
\newtheorem{theorem}[thm]{Theorem}

\newtheorem{prop}[thm]{Proposition}
\newtheorem{claim}[thm]{Claim}
\newtheorem{fact}[thm]{Fact}

\newtheorem{lemma}[thm]{Lemma}

\theoremstyle{definition}
\newtheorem{rem}[thm]{Remark}

\newtheorem{conj}[thm]{Conjecture}
\newtheorem{prob}[thm]{Problem}

\usepackage{accents}

\let\phi=\varphi

\begin{document}

\title[Variations on twins in permutations]{Variations on twins in permutations}

\author{Andrzej Dudek}
\address{Department of Mathematics, Western Michigan University, Kalamazoo, MI, USA}
\email{\tt andrzej.dudek@wmich.edu}
\thanks{The first author was supported in part by Simons Foundation Grant \#522400.}

\author{Jaros\l aw Grytczuk}
\address{Faculty of Mathematics and Information Science, Warsaw University of Technology, Warsaw, Poland}
\email{j.grytczuk@mini.pw.edu.pl}
\thanks{The second author was supported in part by the Polish NSC grant 2015/17/B/ST1/02660.}

\author{Andrzej Ruci\'nski}
\address{Department of Discrete Mathematics, Adam Mickiewicz University, Pozna\'n, Poland}
\email{\tt rucinski@amu.edu.pl}
\thanks{The third author was supported in part by the Polish NSC grant 2018/29/B/ST1/00426}

\begin{abstract}
Let $\pi$ be a permutation of the set $[n]=\{1,2,\dots, n\}$. Two disjoint order-isomorphic subsequences of $\pi$ are called \emph{twins}. How long twins are contained in every permutation? The well known Erd\H {o}s-Szekeres theorem implies that there is always a pair of twins of length $\Omega(\sqrt{n})$. On the other hand, by a simple probabilistic argument Gawron proved that for every $n\geqslant 1$ there exist permutations with all twins having length $O(n^{2/3})$. He conjectured  that the latter bound is the correct size of the longest twins guaranteed in every permutation. We support this conjecture by showing that almost all permutations contain twins of length  $\Omega(n^{2/3}/\log n^{1/3})$. Recently, Bukh and Rudenko have tweaked our proof and removed the log-factor. For completeness, we also present our version of their proof (see Remark \ref{BB} below on the interrelation between the two proofs).

 In addition, we study several variants of the problem with diverse restrictions imposed on the twins. For instance, if we restrict attention to twins avoiding a fixed permutation $\tau$, then the corresponding extremal function equals $\Theta(\sqrt{n})$, provided that $\tau$ is not monotone. In case of \emph{block twins} (each twin occupies a segment) we prove that it is $(1+o(1))\frac{\log n}{\log\log n}$, while for random permutations it is twice as large. For twins that jointly occupy a segment (\emph{tight twins}), we prove that for every $n$ there are permutations avoiding them on all segments of length greater than $24$.

\end{abstract}

\maketitle


\setcounter{footnote}{1}

\section{Introduction}

Looking for twin objects in mathematical structures has long and rich tradition, going back to some geometric dissection problems that culminated in the famous Banach-Tarski Paradox (see \cite{TomkowiczWagon}). A general problem is to split a given structure (or a pair of structures) into few pairwise isomorphic substructures. A related question is: \emph{How large disjoint isomorphic substructures can be found in a given structure?}

In this paper we study this question for permutations. To put our work in a broader context, we mention two similar problems: for graphs and for sequences. More can be found in a  survey by Axenovich \cite{Axenovich}.

\subsection{Twins in graphs}
Ulam,  inspired (almost surely) by the famous Banach-Tarski Paradox (see \cite{TomkowiczWagon}), asked (see \cite{Graham}) the following question. Given a pair of graphs $G$ and $H$ with the same order and size, what is the least integer $k=U(G,H)$ such that the edges of $G$ and $H$ can be partitioned into  edge-disjoint subgraphs $G_1,\dots,G_k$ of $G$ and $H_1,\dots, H_k$ of $H$ such that $G_i$ is isomorphic to $H_i$ for every $i=1,2,\dots,k$? Let $U(n)$ denote the maximum of $U(G,H)$ over all pairs of graphs on $n$ vertices with the same number of edges. It was proved by Chung, Graham, Erd\H {o}s, Ulam, and Yao in \cite{ChungEGUY}, that $$U(n)=\frac{2}{3}n+o(n).$$More general results and some open problems can be found in \cite{ChungEG}.

A related problem is to find two large isomorphic subgraphs in two given graphs, or in one given graph. Let $f(m)$ denote the largest integer $k$ such that every graph with $m$ edges contains a pair of \emph{twins}, that is, two edge-disjoint isomorphic subgraphs with $k$ edges each. The problem was stated independently by Jacobson and Sch\"{o}nheim (see \cite{ErdosPachPyber}). Currently, the best general result, due to Lee, Loh, and Sudakov \cite{LeeLohSudakov}, states that $$f(m)=\Theta(m\log m)^{2/3}.$$ In \cite{AlonCaroKrasikov}, Alon, Caro, and Krasikov proved that every tree with $m$ edges contains a pair of twins of total size at least $$m-O\left(\frac{m}{\log\log m}\right).$$

\subsection{Twins in sequences}
By \emph{twins} in a sequence over an alphabet we mean a pair of identical subsequences with disjoint sets of indices. Let $g_r(n)$ denote the maximum length of twins in every sequence of length $n$ over an alphabet with $r$ symbols. Axenovich, Person, and Puzynina proved in \cite{AxenovichPersonPuzynina} that $$g_r(n)\geqslant \frac{1}{r}n-o(n).$$ This result is particularly surprising for $r=2$, as it says that every binary sequence consists of two identical subsequences plus an asymptotically negligible part. The proof is based on a new regularity lemma for sequences. Currently, the best lower bound for $g_r(n)$, obtained by Bukh and Zhou \cite{BukhZ}, asserts that $$g_r(n)\geqslant cr^{-2/3}n-o(n),$$ for some constant $c>0$. It was also proved in \cite{BukhZ} that $g_4(n)\leqslant 0.4932n$. The case of ternary sequences remains open.

\subsection{Twins in permutations}
By a \emph{permutation} we mean  any finite sequence of  distinct positive integers. We say that two permutations $(x_1,\dots,x_k)$ and $(y_1,\dots,y_k)$ are \emph{similar} if their entries preserve the same relative order, that is, $x_i<x_j$ if and only if $y_i<y_j$ for all pairs $\{i,j\}$ with $1\leqslant i<j\leqslant k$. Note that given a permutation $(x_1,\dots,x_k)$ and a $k$-element set $y$ of positive integers, there is only one permutation of $y$ similar to $(x_1,\dots,x_k)$.

Let $[n]=\{1,2,\dots,n\}$ and $\pi$ be a permutation of $[n]$, called also an \emph{$n$-permutation}. Two similar disjoint sub-permutations of $\pi$ are called \emph{twins} and the \emph{length of a pair of twins} is defined as the number of elements in just one of the sub-permutations. For example, in permutation $$(6,\colorbox{cyan}{1},\colorbox{cyan}{4},7,\colorbox{Lavender}{3},9,\colorbox{Lavender}{8},\colorbox{cyan}{2},\colorbox{Lavender}{5}),$$
the blue $(1,4,2)$ and red $(3,8,5)$ subsequences form a pair of twins of length $3$, both similar to  $(1,3,2)$.

Let $t(\pi)$ denote the largest integer $k$ such that $\pi$ contains a pair of twins of length $k$. Let $t(n)$ denote the minimum of $t(\pi)$ over all permutations $\pi$ of $[n]$. In other words, $t(n)$ is the largest integer $k$ such that every $n$-permutation contains a pair of twins of length $k$. Our aim is to estimate this function, as well as some of its variants subject to various restrictions.

By the classical result of Erd\H {o}s and Szekeres \cite{ErdosSzekeres} concerning monotone subsequences of permutations, we get $t(n)= \Omega(\sqrt{n})$. Indeed, any splitting of a monotone sequence into two subsequences of the same length gives a pair of twins. On the other hand, using a probabilistic argument, Gawron \cite{Gawron} proved that $t(n)=O(n^{2/3})$. He also made a conjecture that the later bound is the correct order of the function $t(n)$.

\begin{conj}[Gawron \cite{Gawron}]\label{Conjecture Gawron}
	We have $t(n)=\Theta(n^{2/3})$.
\end{conj}
\noindent
Very recently, Bukh and Rudenko~\cite{BukhR} showed that $t(n)=\Omega(n^{3/5})$.

Our main result  supports this conjecture and states that a random permutation $\Pi_n$, selected uniformly  from all $n!$ permutations of $[n]:=\{1,2,\dots,n\}$, satisfies $$t\left(\Pi_n\right)= \Omega\left(\frac{n^{2/3}}{\log^{1/3}n}\right)$$
asymptotically almost surely (a.a.s).

\begin{rem}\label{BB}
After our manuscript was uploaded to arXiv, we learned from Bukh and Rudenko that our argument could be tweaked to drop the log term and yield the optimal lower bound
$t\left(\Pi_n\right)= \Omega(n^{2/3}).$ The new idea delivered by them was to replace estimation of the maximum degree of the auxiliary bipartite graph $B$ by the estimation of the average degree (see the proofs below for details). Bukh and Rudenko wrote then a short note ~\cite{BukhR} which contains the improvement. In their version they switched to a less standard way of generating  a random permutation by a Poisson point process on the unit square. Meanwhile, we realized that their idea can be implemented in a more elementary way avoiding the poissonization.
  In face of all these circumstances, in the current version of the manuscript we decided to include this new result, Theorem~\ref{rg2}, as well.
\end{rem}

\bigskip

We also consider several variants of the function $t(n)$ obtained by imposing various restrictions on the structure or position of twins in  permutations. Let $\tau$ be a fixed permutation. We say that a permutation $\sigma$ \emph{avoids} $\tau$ if there is no sub-permutation of $\sigma$ similar to $\tau$. Let $t(n,\tau)$ be the largest integer $k$ such that every $n$-permutation contains a pair of $\tau$-avoiding twins of length $k$. Using the celebrated result of Marcus and Tardos \cite{MarcusTardos}, confirming the Stanley-Wilf Conjecture, we prove (Theorem \ref{Theorem Forbidden Pattern}) that $$t(n,\tau)=\Theta(\sqrt{n}),$$provided $\tau$ is non-monotone.

Two further variants set  restrictions on the occurrence of twins in  permutations. By \emph{block twins} in $\pi$ we mean a pair of twins, each occupying a segment of consecutive terms of $\pi$. For instance, in permutation $(6,\colorbox{Lavender}{5},\colorbox{Lavender}{2},\colorbox{Lavender}{3},8,9,\colorbox{cyan}{7},\colorbox{cyan}{1},\colorbox{cyan}{4})$, the red $(5,2,3)$ and blue $(7,1,4)$ subsequences form a pair of block twins similar to permutation $(3,1,2)$. Let $bt(n)$ be the largest size of block twins one can find in \emph{every} $n$-permutation. We prove  (Theorem \ref{Theorem Block Twins}) that $$bt(n) = (1+o(1)) \frac{\log n}{\log \log n}.$$ Interestingly, for a random permutation $\Pi_n$ the analogous function is twice as large (Theorem \ref{Theorem Block Twins Random}).

If a pair of twins jointly occupies a segment in $\pi$, then we call them \emph{tight twins}. For example, in  permutation $(6,\colorbox{Lavender}{5},\colorbox{cyan}{7},\colorbox{cyan}{1},\colorbox{Lavender}{2},\colorbox{Lavender}{3},\colorbox{cyan}{4},9,8)$, the red $(5,2,3)$ and blue $(7,1,4)$ subsequences form a pair of tight twins similar to permutation $(3,1,2)$. Let $tt(n)$ be the largest size of tight twins one can find in \emph{every} $n$-permutation. By using the Lov\'asz Local Lemma we prove (Theorem \ref{Theorem Tight Twins}) that $tt(n)\leqslant 12$ for all $n\geqslant 1$ which means that there exist permutations avoiding tight twins of length at least $13$. On the other hand, in Proposition \ref{Proposition Tight Twins} we demonstrate that every permutation of six elements contains tight twins of length at least $2$.

Combining the last two variants of the problem, one may consider the most restrictive \emph{tight block twins}, which are block twins jointly occupying a segment. For instance, in permutation
$(6,9,\colorbox{Lavender}{5},\colorbox{Lavender}{2},\colorbox{Lavender}{3},\colorbox{cyan}{7},\colorbox{cyan}{1},\colorbox{cyan}{4},8)$, the red $(5,2,3)$ and blue $(7,1,4)$ subsequences form a pair of tight block twins similar to permutation $(3,1,2)$. Surprisingly, as proved by Avgustinovich, Kitaev, Pyatkin, and Valyuzhenich in \cite{AvgustinovichKPV}, for every $n$ there exist $n$-permutations  containing no such twins of  length more than one. This result constitutes a permutation counterpart of the famous theorem of Thue \cite{Thue} on \emph{non-repetitive sequences} which asserts that there exist arbitrarily long ternary sequences avoiding tight block twins of any possible length (see \cite{Lothaire}).

\bigskip
\textbf{Acknowledgments.}
We are very grateful to Boris Bukh and Oleksandr Rudenko for their interest in our work  and a  fruitful discussion.

\section{General twins}

In this section we will prove our main result on twins in random permutations. Let $|\sigma|$ denote the length of a permutation $\sigma$. Recall that $t(\pi)$ denotes the maximum length of twins in a permutation $\pi$, that is,
$$t(\pi)=\max\{\text{$|\sigma_1|:(\sigma_1,\sigma_2)$ is a pair of twins in $\pi$}\},$$
and $t(n)$ is defined as $$t(n)=\min\{t(\pi):\text{$\pi$ is a permutation of $[n]$}\}.$$

For completeness, we begin with reproducing the result of Gawron.
\begin{thm}[Gawron \cite{Gawron}]\label{thm:gawron}
We have $$\Omega (\sqrt{n}) =  t(n) =  O(n^{2/3}).$$
\end{thm}

\begin{proof}The lower bound follows immediately from the well known theorem of Erd\H {o}s and Szekeres \cite{ErdosSzekeres} asserting that every permutation of length $n$ contains a monotone subsequence of length $\Omega(\sqrt{n})$.
	
For the upper bound we use the first moment method. Let $\Pi:=\Pi_n$ be a random permutation chosen uniformly from the set of all $n!$ permutations of $[n]$. Let $k$ be a fixed positive integer and let $X$ be a random variable counting all pairs of twins of length $k$ in $\Pi$. Furthermore, for a pair of disjoint subsequences $s,t$ in $[n]$, each of length $k$, let $X_{s,t}$ be an indicator random variable such that $X_{s,t}=1$ if there is a pair of twins in $\Pi$ on subsequences $s$ and $t$. So, $X=\sum_{s,t} X_{s,t}$ and by the linearity of expectation $$\EE X=\sum_{s,t}\EE X_{s,t}=\sum_{s,t}\PP(X_{s,t}=1).$$ Since
\begin{equation}\label{1|k!}
\PP(X_{s,t}=1)=\frac{\binom nk\cdot (n-k)!\cdot 1}{n!}=\frac{1}{k!}
\end{equation}
 and the number of unordered pairs $\{s,t\}$ is $\frac{1}{2}\binom{n}{2k}\binom{2k}{k}$,
  $$\EE X=\frac{1}{2}\binom{n}{2k}\binom{2k}{k}\frac{1}{k!}=\frac{n(n-1)\dots (n-2k+1)}{2(k!)^3}.$$ Using the inequality $k!>\frac{k^k}{e^k}$ gives \begin{equation}\label{<1}
  \EE X<\frac{n^{2k}e^{3k}}{2k^{3k}}<\left(\frac{n^2e^3}{k^3}\right)^k.
   \end{equation}
   It follows that for $k>en^{2/3}$ we have $\EE X<1$, which means that there must be an $n$-permutation $\pi$ with $X(\pi)=0$, that is,  with no twins of length $k$. This completes the proof.
\end{proof}

One could naturally hope to improve the upper bound for $t(n)$ by some more refined probabilistic tools. In fact, in his master thesis \cite{Gawron}, Gawron made such an attempt by using the Lov\'{a}sz Local Lemma. However, the resulting bound turned out to be the same (up to a constant).

\subsection{General twins in random permutations} In view of Theorem \ref{thm:gawron}, to prove Conjecture \ref{Conjecture Gawron} it is enough to show that $t(n)=\Omega(n^{2/3})$.
In this subsection we will prove that this bound holds for \emph{almost} all permutations.

Let $\Pi:=\Pi_n$ be a random permutation of $[n]$ and let $t(\Pi)$ be the corresponding random variable equal to the maximum length of twins in $\Pi$. Recall that by saying that some property of a random object holds \emph{asymptotically almost surely} (a.a.s.~for short) we mean that it holds with probability tending to one with the size of the object growing to infinity.

\begin{thm}\label{rg}
For a random permutation $\Pi_n$, a.a.s.
\[
\Omega\left(\frac{n^{2/3}}{\log^{1/3}n}\right) = t(\Pi_n) = O(n^{2/3}).
\]
\end{thm}

 After proving Theorem \ref{rg}, we will turn to present a proof of the optimal bound that was  recently obtained by Bukh and Rudenko~\cite{BukhR} (see Remark \ref{BB}).
 
\begin{thm}[\cite{BukhR}]\label{rg2}
For a random permutation $\Pi_n$, a.a.s.
\[
t(\Pi_n) = \Theta(n^{2/3}).
\]
\end{thm}

\begin{proof}[Proof of Theorem~\ref{rg}]
The upper bound  follows immediately from the proof of Theorem~\ref{thm:gawron}. Indeed, by (\ref{<1}) with $k = 2en^{2/3}$  we have
$\PP(X\ge1)\le\E X<2^{-3k}\to0$ as $n\to\infty$. 

Now we proceed with a much more challenging proof of the lower bound.
Set $a=(Cn\log n)^{1/3}$, assume for convenience that $a$ divides $n$, and partition $[n]$ into $n/a$ consecutive blocks of equal size, that is, set
$$[n]=A_1\cup\cdots\cup A_{n/a},$$
where $|A_1|=\cdots=|A_{n/a}|=a$. For fixed $1\le i,j\le n/a$,
let $X:=X_{ij}$ be the number of elements from the set $A_j$ which $\Pi$ puts on the positions belonging to the set $A_i$. We construct an auxiliary $n/a\times n/a$ bipartite graph $B$ with vertex classes $U=\{1,\dots,n/a\}$ and $V=\{1,\dots, {n/a}\}$, where $ij\in B$ iff $X_{ij}\ge 2$.

 Let $M=\{i_1j_1,\dots, i_mj_m\}$, $i_1<\cdots<i_m$, be a matching in $B$ of size $|M|=m$. For every $ij\in M$, let $s_i,t_i$ be some two elements of $A_i$ such that $\Pi(s_i),\Pi(t_j)\in A_j$. Then, $\{s_{i_1},\dots,s_{i_m}\}$ and
$\{t_{i_1},\dots,t_{i_m}\}$ form a pair of twins. Indeed, if say $\Pi(s_{i_1})<\Pi(s_{i_2})$, then $j_1<j_2$, and so $\Pi(t_{i_1})<\Pi(t_{i_2})$.

Hence, it remains to show that a.a.s.~there is a matching in $B$ of size $m=\Omega(n^{2/3}/\log^{1/3}n)$. To this end, we are going to use the obvious fact, coming from a greedy algorithm, that in every graph $G$ there is a matching of size at least
\begin{equation}\label{2Delta}
|E(G)|/(2\Delta(G)),
\end{equation}
where $\Delta(G)$ is the maximum vertex degree in $G$.

Our plan is to first estimate the probability of an edge in $B$, that is, $\PP(X_{ij}\ge2)$, and then, apply the inequality of Talagrand to show that the degrees in $B$ are tightly concentrated around their means, that is, around $(n/a)\PP(X_{ij}\ge2)$. Then, the ratio $|E(B)|/(2\Delta(B))$ will easily be estimated.

\begin{fact}\label{Xij2} We have
\[\PP(X_{ij}\ge2)  \sim \frac{a^4}{2n^2}.\]
\end{fact}

\begin{proof} Notice that
\[
\PP(X_{ij}=0)=\frac1{n!}\binom{n-a}aa!(n-a)!=\frac{\binom{n-a}{a}}{\binom{n}{a}}
\]
and
\[
\PP(X_{ij}=1)=\frac1{n!}\binom{n-a}{a-1}\binom{a}{1}^2(a-1)!(n-a)!
=\frac{a^2}{n-2a+1} \frac{\binom{n-a}{a}}{\binom{n}{a}}
\sim\frac{a^2}{n} \frac{\binom{n-a}{a}}{\binom{n}{a}}.
\]
Since $\binom{m}{k} \sim \frac{m^k}{k!}$ if $k^2 = o(m)$, we obtain
\[
\frac{\binom{n-a}{a}}{\binom{n}{a}}
\sim \left(\frac{n-a}{n}\right)^a
= \left(1-\frac{a}{n}\right)^a
\sim \exp(-a^2/n)
\sim 1-\frac{a^2}n+\frac{a^4}{2n^2} + o(a^4/n^2)
\]
and hence
\[
\PP(X_{ij}=0)+\PP(X_{ij}=1)
\sim \left(1-\frac{a^2}n+\frac{a^4}{2n^2} + o(a^4/n^2)\right)\left(1+\frac{a^2}{n}\right)
\sim 1-\frac{a^4}{2n^2},
\]
which is equivalent to the statement of Fact~\ref{Xij2}.
\end{proof}

We continue with the proof of Theorem \ref{rg}. Let $Y:=Y_i=\sum_{j=1}^{n/a}\mathbb{I}(X_{ij}\ge2)$. We are going to apply to $Y$ a concentration inequality which follows from some more general results in \cite{Talagrand}. Here is a slightly simplified version from \cite{LuczakMcDiarmid} (see also \cite{McDiarmid}).

\begin{theorem}[Luczak and McDiarmid \cite{LuczakMcDiarmid}]\label{tala}
 Let $h(\pi)$ be a function of $n$-permutations which, for some positive constants $c$ and $r$, satisfies
\begin{enumerate}[label=\rmlabel]
\item\label{thm:talagrand:i} if $\pi_2$ is obtained from $\pi_1$ by swapping two elements, then $|h(\pi_1)-h(\pi_2)|\le c$;
\item\label{thm:talagrand:ii} for each $\pi$ and $s>0$, if $h(\pi)=s$, then in order to show that $h(\pi)\ge s$, one needs to
specify only at most $rs$ values $\pi(i)$.
\end{enumerate}
Then, for every $\eps>0$,
$$\PP(|h(\Pi)-m|\ge \eps m)\le4 \exp(-\eps^2 m/(32rc^2)),$$
where $m$ is the median of $h(\Pi)$.
\end{theorem}

Observe that $Y=Y(\Pi)$ satisfies  assumptions \ref{thm:talagrand:i} and \ref{thm:talagrand:ii} with $c=1$ and $r=2$, respectively. Indeed, swapping two elements of $\Pi$ changes $\mathbb{I}(X_{ij}\ge2)$ for at most one value of $j$. Moreover, to exhibit the event $Y\ge s$, it is sufficient to reveal $2s$ values of $\Pi$. Let $m$ be the median of $Y$. Then, by Theorem \ref{tala},
$$\PP(|Y-m|\ge m/2)\le4\exp(-m/256).$$
Moreover, it follows (see for example \cite{Talagrand}, Lemma 4.6, or \cite{McDiarmid}, page 164)  that $|\E Y-m|=O(\sqrt m)$. In particular,
$$\PP(|Y-\E Y|\ge(2/3)\E Y)\le 4\exp(-\E Y/300).$$
Note that
$$\E Y_i=\sum_{j=1}^{n/a}\PP(X_{ij}\ge2)\sim \frac na\cdot \frac{a^4}{2n^2}=\frac{a^3}{2n}=(C/2)\log n,$$ so, for $C$ large enough, the union bound implies that a.a.s. for all $i=1,\dots,n/a$, we have
$$
\frac{a^3}{7n} \le \frac13\E Y_i\le Y_i\le \frac53\E Y_i\le \frac{a^3}{n}.$$
This implies that two further facts hold a.a.s.:
$$|E(B)|= \sum_{i=1}^{n/a}Y_i\ge\frac n{3a}\E Y\ge\frac{a^2}7$$
and
$$\Delta(B)\le\max Y_i\le \frac{a^3}{n}.$$
Finally, the largest matching in $B$ has size at least
$$\frac{|E(B)|}{2\Delta(B)}\ge \frac{n}{14a}=\Omega\left(\frac{n^{2/3}}{\log^{1/3}n}\right).$$
This completes the proof of Theorem \ref{rg}.
\end{proof}

Now we show, based on \cite{BukhR}, how to modify the previous proof obtaining the optimal lower bound $\Omega(n^{2/3})$. One difference is that instead of the Talagrand inequality applied to the degrees in~$B$, we invoke Azuma-Hoeffding inequality directly to our principal random variable $t(\Pi_n)$. Another modification lies in using inequality  (\ref{2Delta}) not for the entire graph $B$, but for its suitable chosen subgraph.

\begin{proof}[Proof of Theorem~\ref{rg2}]
Let $C\ge 3$. Set $a=(Cn)^{1/3}$, assume for convenience that $a$ divides~$n$, and as in the previous proof we construct the auxiliary $n/a\times n/a$ bipartite graph $B$. We are going to use~\eqref{2Delta} again, however, this time we will apply this fact to a subgraph of $B$ obtained after deleting vertices of  large degrees.

For each $i\in U\cup V$, let $Y_i=\sum_{j=1}^{n/a}\mathbb{I}(X_{ij}\ge2)$ be the degree of vertex $i$ in $B$. Let $\nu(B)$ be the size of a largest matching in $B$. Note that $\Delta(B)\le a/2$. Further, let $Z_k$ be the number of vertices of degree $k$ in $B$, $k=0,\dots,a/2$, and set $C'=\lceil e^2 C/2 \rceil$. Then, by applying inequality (\ref{2Delta}) to the subgraph $B'$ of $B$ obtained by deleting all vertices of degree at least $C'$, we get
\[
\nu(B)\ge\nu(B')\ge\frac{|E(B')|}{2\Delta(B')}\ge\frac{|E(B)|-\sum_{k=C'}^{a/2}kZ_k}{2C'}.
\]
Taking expectation on both sides and noticing that $\E Z_k=(2n/a)\PP(Y_1=k)$, we infer that
\begin{equation}\label{EnuB}
\E[\nu(B)]\ge\frac1{2C'}\left(\E[|E(B)|]-(2n/a)\sum_{k=C'}^{a/2}k\PP(Y_1=k)\right).
\end{equation}
Observe that, trivially, $|E(B)|$ is at least as large as the number of vertices $i\in U$ with positive degree. Thus, to estimate $\E(|E(B)|)$  from below it suffices to estimate $\PP(Y_i=0)$.
By  using  standard approximations
\[
\binom{m}{k} \sim \frac{m^k}{k!} \text{ if } k^2 = o(m) \quad \text{ and } \quad
\binom{m}{k} \sim \frac{m^k}{k!}\exp\left( -\frac{k^2}{2m} - \frac{k^3}{6m^2}\right) \text{ if } k = o(m^{3/4}),
\]
we get
\begin{align*}
\PP(Y_i=0) &= \binom{\frac{n}{a}}{a} \binom{a}{1}^a a! (n-a)! \cdot \frac{1}{n!}
= \binom{\frac{n}{a}}{a} a^a \cdot \frac{1}{\binom{n}{a}}\\
&\sim \frac{\left(\frac{n}{a}\right)^a}{a!} \exp\left( -\frac{a^2}{2n/a} - \frac{a^3}{6(n/a)^2}\right)
a^a \cdot \frac{a!}{n^a} \sim \exp\left( -\frac{a^3}{2n}\right) = e^{-C/2} < \frac{1}{2}
\end{align*}
and so
\[
\E[|E(B)|] \ge \frac{n}{a} \PP(Y_1\ge 1) \ge \frac{n}{2a} = \frac{n^{2/3}}{2C^{1/3}}.
\]

Now we estimate $\PP(Y_1=k)$ for $k\in \{C',\dots,a/2\}$. Observe that
\begin{align*}
\PP(Y_1 = k) &\le \binom{\frac{n}{a}}{k} \binom{a}{2}^k \binom{a}{2k} (2k)! (n-2k)! \cdot \frac{1}{n!}
= \binom{\frac{n}{a}}{k} \binom{a}{2}^k \binom{a}{2k} \cdot \frac{1}{\binom{n}{2k}}\\
&\le \frac{\left(\frac{n}{a}\right)^k}{k!} \frac{a^{2k}}{2^k} \frac{a^{2k}}{(2k)!} \frac{(2k)!}{n^{2k}}
= \frac{a^{3k}}{2^k n^k k!} \le \frac{a^{3k}}{2^k n^k (k/e)^k}
= \exp\left( -k \log\left( \frac{2kn}{ea^3}\right)\right) \le e^{-k},
\end{align*}
since $k\ge C' \ge e^2 C/2$. Thus, since $C'\ge 4$,
\[
\sum_{k=C'}^{a/2}k\PP(Y_1=k)  \le \sum_{k=C'}^{a/2}ke^{-k} \le \sum_{k=C'}^{\infty}ke^{-k}  = \frac{C'(1-1/e)+1/e}{(1-1/e)^2}e^{-C'} \le \frac{1}{8}.
\]

Finally, returning to~\eqref{EnuB}, we conclude that
\[
\E[\nu(B)]\ge\frac{n^{2/3}}{2C'C^{1/3}}\left(\frac{1}{2} - 2\cdot \frac{1}{8} \right)=\frac{n^{2/3}}{8C'C^{1/3}} = \Omega(n^{2/3}).
\]
Since $t(\Pi_n)\ge\nu(B)$, to complete the proof of Theorem~\ref{rg2} it remains to show that $t(\Pi_n)$ is highly concentrated about its mean.

For this, we are going to use the Azuma-Hoeffding inequality for random permutations (see, e.g., Lemma 11 in~\cite{FP} or  Section 3.2 in~\cite{McDiarmid98}):
\begin{theorem}\label{azuma}
 Let $h(\pi)$ be a function of $n$-permutations such that if permutation $\pi_2$ is obtained from permutation $\pi_1$ by swapping two elements, then $|h(\pi_1)-h(\pi_2)|\le 1$.
Then, for every $\eta>0$,
\[
\PP(|h(\Pi_n)-\E[h(\Pi_n)]|\ge \eta)\le 2\exp(-\eta^2/(2n)).
\]
\end{theorem}

 To verify the Lipschitz assumption, note that if $\pi_2$ is obtained from a permutation $\pi_1$ by swapping any two of its  elements, then $|t(\pi_1)-t(\pi_2)|\le2$. Indeed, let $\pi_2'$ be obtained from $\pi_2$ by removing the elements that were swapped. Then, clearly, $\pi_2'$ is a sub-permutation of $\pi_1$ and $t(\pi_1) \ge t(\pi_2') \ge t(\pi_2)-2$. Similarly, one can argue that $t(\pi_2) \ge  t(\pi_1)-2$.
Consequently, Theorem~\ref{azuma} applied with $h(\pi)=t(\pi)/2$ and $\eta=n^{3/5}$ implies
\[
\PP(|t(\Pi_n)-\E[t(\Pi_n)]|\ge n^{3/5})=o(1),
\]
finishing the proof.
\end{proof}

\subsection{Twins with a forbidden pattern}
In this subsection we prove tight asymptotic bounds for the maximum length of twins avoiding a fixed non-monotone permutation $\tau$. Recall that a permutation $\sigma$ \emph{avoids} $\tau$, or is \emph{$\tau$-free}, if no subsequence of $\sigma$ is similar to $\tau$. We will need the following result of Marcus and Tardos \cite{MarcusTardos} confirming a famous conjecture stated independently by Stanley and Wilf (see \cite{Stanley}).

\begin{theorem}[Marcus and Tardos \cite{MarcusTardos}]\label{Theorem Marcus-Tardos}
	Let $\tau$ be a fixed permutation and let $d_{\tau}(k)$ denote the number of $k$-permutations avoiding $\tau$. Then there exists a positive constant $c$, depending only on $\tau$, such that $d_{\tau}(k)\leqslant c^k$.
\end{theorem}

Let $t(\pi,\tau)$ denote the maximum length of $\tau$-free twins in a permutation $\pi$:
$$t(\pi,\tau)=\max\{\text{$|\sigma_1|:(\sigma_1,\sigma_2)$ is a pair of $\tau$-free twins in $\pi$}\},$$
and let $t(n,\tau)$ be defined by $$t(n,\tau)=\min\{t(\pi,\tau):\text{$\pi$ is a permutation of $[n]$}\}.$$

\begin{thm}\label{Theorem Forbidden Pattern} Let $\tau$ be a non-monotone permutation. Then
\[
t(n,\tau) = \Theta(\sqrt{n}).
\]
\end{thm}
\begin{proof}
Owing to the non-monotonicity of $\tau$, the lower bound follows from the Erd\H{o}s-Szekeres theorem in the same way as the bound $t(n)=\Omega(\sqrt n)$ in Theorem \ref{thm:gawron}. For the upper bound, let $\Pi$ be a random permutation of $[n]$, and let $X_{\tau}$ denote the random variable counting the number of $\tau$-free twins of size $k$ in $\Pi$. Further, let $d_\tau$ be the number of $\tau$-free permutations of $[k]$.
We have
\[
\E(X_\tau) = \frac{1}{2} \binom{n}{k}\binom{n-k}{k} \cdot p_\tau,
\]
where
\[
p_\tau = \frac{\binom{n}{k} d_\tau \cdot \binom{n-k}{k}\cdot  1 \cdot (n-2k)!}{n!} = \frac{d_\tau}{(k!)^2}.
\]
 By Theorem \ref{Theorem Marcus-Tardos}, $p_\tau \le c^k / (k!)^2$, which implies
\begin{multline*}
\E(X_\tau) \le \frac{1}{2} \binom{n}{k}\binom{n-k}{k} \frac{c^k}{(k!)^2}
= \frac{1}{2} \frac{n! c^k}{(k!)^4 (n-2k)!}\\
= \frac{1}{2} \frac{(n)_{2k} c^k}{(k!)^4 }
\le \frac{1}{2} \frac{n^{2k} c^k}{(k!)^4 }
\le \frac{1}{2} \frac{n^{2k} c^{k}}{(k/e)^{4k} }
= \frac{1}{2} \left( \frac{n c^{1/2}}{(k/e)^{2} } \right)^{2k} < 1
\end{multline*}
for $k\ge ec^{1/4}\sqrt{n}$. Thus, $t(n,\tau)=O(\sqrt n)$ which completes the proof.
\end{proof}

\section{Variations}

In this part of the paper we consider twins with some restrictions on positions they occupy in a permutation. We focus on \emph{blocks} in permutations, that is, subsequences whose index sets are segments of consecutive integers. As a key technical tool we are going to use two versions of the Local Lemma which we state first.

\subsection{Two versions of the Local Lemma}

In the next subsection, we will make use of the following symmetric version of the Lov\'{a}sz Local Lemma \cite{ErdosLovasz} (see \cite{AlonSpencer}).
For   events $\cE_{1},\ldots ,\cE_{n}$ in any probability space, \emph{a dependency graph} $D=([n],E)$ is any graph on vertex set $[n]$ such that for every vertex $i$ the event $\cE_i$ is jointly independent of all events $\cE_j$ with $ij\not\in E$.

\begin{lemma}[The Local Lemma; Symmetric Version \cite{ErdosLovasz} (see \cite{AlonSpencer})]\label{LLL Symmetric}
	Let $\cE_{1},\ldots ,\cE_{n}$ be events in any probability space. Suppose that the maximum degree of a dependency graph of these events is at most $\Delta$, and $\PP(A_i)\leqslant p$, for all $i=1,2,\dots,n$. If $ep(\Delta+1)\leqslant 1$, then $\PP \left(
	\bigcap\limits_{i=1}^{n}\overline{\cE_{i}}\right) >0$.
\end{lemma}

In another proof it will be convenient to use the following version of the Lov\'{a}sz Local Lemma, which is equivalent to the standard asymmetric version (see \cite{AlonSpencer}).

\begin{lemma}[The Local Lemma; Multiple Version (see \cite{AlonSpencer})]\label{LLL}
	Let $\cE_{1},\ldots ,\cE_{n}$ be events in any probability space with a dependency
	graph $D=(V,E)$. Let $V=V_{1}\cup \cdots \cup V_{t}$ be a partition such
	that all members of each part $V_{r}$ have the same probability $p_{r}$.
	Suppose that the maximum number of vertices from $V_{s}$ adjacent to a
	vertex from $V_{r}$ is at most $\Delta _{rs}$. If there exist real numbers $%
	0\leq x_{1},\ldots ,x_{t}<1$ such that $p_{r}\leq
	x_{r}\prod\limits_{s=1}^{t}(1-x_{s})^{\Delta _{rs}}$, then $\Pr \left(
	\bigcap\limits_{i=1}^{n}\overline{\cE_{i}}\right) >0$.
\end{lemma}

\subsection{Block twins}
A pair of twins $(\sigma_1,\sigma_2)$ in a permutation $\pi$ is called  \emph{block twins} if both sub-permutations, $\sigma_1$ and $\sigma_2$, are blocks of $\pi$. Let $bt(\pi)$ denote the largest length of block twins in $\pi$, that is,
$$bt(\pi)=\max\{\text{$|\sigma_1|:(\sigma_1,\sigma_2)$ is a pair of block twins in $\pi$}\},$$
and let $$bt(n)=\min\{bt(\pi):\text{$\pi$ is a permutation of $[n]$}\}.$$
The proofs in this and the next subsection rely on the following simple fact about the probability of appearance of $r$ pairs of twins on fixed positions in a random permutation.

\begin{fact}\label{ABCD}
For $r\ge2$ let $A_i,B_i$, $i=1,\dots,r$ be $k$-elements segments of $[n]$ with $A_i\cap B_i=\emptyset$ as well as $A_1\cap\bigcup_{i=2}^r(A_i\cup B_i)=\emptyset$.
Further, let $\cE_i$ be the event that the pair $(A_i,B_i)$ induces block twins in $\Pi$. Then,
$$\PP(\cE_1\cap\cdots\cap\cE_r)=\PP(\cE_1)\PP(\cE_2\cap\cdots\cap\cE_r).$$
\end{fact}

\proof By (\ref{1|k!}), $\PP(\cE_1)=1/k!$. Let $N$ be the number of permutations of the set $[n]\setminus A_1$ such that all pairs $A_i,B_i$, $i=2,\dots,r$ span block twins. Observe that
$$|\cE_1\cap\cdots\cap\cE_r|=\binom nkN\quad\mbox{and}\quad|\cE_2\cap\cdots\cap\cE_r|=(n)_kN,$$
where the first equality follows from the fact that once  the values of $\Pi(i)$ are fixed on $[n]\setminus A_1$, the rest of $\Pi$ is determined.
Hence,
$$\PP(\cE_1\cap\cdots\cap\cE_r)=\frac{\binom nkN}{n!}=\frac{\binom nk|\cE_2\cap\cdots\cap\E_r|}{(n)_kn!}=\frac1{k!}\PP(\cE_2\cap\cdots\cap\cE_r).\qed$$

\medskip

The following result gives an asymptotic formula for the function $bt(n)$.

\begin{thm}\label{Theorem Block Twins}
We have
\[
bt(n) = (1+o(1)) \frac{\log n}{\log \log n}.
\]
\end{thm}

\begin{proof}
First we show the lower bound. Let $n = k(k! +1)$ and let $\pi$ be any permutation of~$[n]$. Divide $\pi$ into $k!+1$ blocks, each of length $k$. By the pigeonhole principle, there are two blocks that induce similar sub-permutations (forming thereby a pair block twins). The choice of $n$, together with the Stirling formula, imply that $k = (1+o(1))\frac{\log n}{\log \log n}$.

For the upper bound we use Lemma \ref{LLL Symmetric} and Fact \ref{ABCD}. Let $n = \frac{(k-1)!}{4e}$ and let $\Pi_n$ be a random permutation.
For a given pair of indices $i$ and $j$ with $1\le i\le j-k\le n-2k$, define the event $\cE_{i,j}$ that subsequences $(\Pi(i),\Pi(i+1)),\dots,\Pi(i+k-1))$ and $(\Pi(j),\Pi(j+1)),\dots,\Pi(j+k-1))$ are block twins. We have $p:=\PP(\cE_{i,j})=1/k!$.

Also notice that by Fact \ref{ABCD}, a fixed event $\cE_{i,j}$  is jointly independent of all the events  $\cE_{i',j'}$ for which
$$\{i,i+1,\dots,i+k-1\}\cap(\{i',i'+1,\dots,i'+k-1\}\cup\{j',j'+1,\dots,j'+k-1\})=\emptyset.$$
Thus, there is a dependency graph $D$ with maximum degree at most
\[
\Delta = 2(2k-1)(n-k) \le 4kn-1.
\]
This and the choice of $n$ gives that
\[
e(\Delta+1)p \le e \cdot 4kn \cdot \frac{1}{k!} = 1.
\]
Thus,  Lemma \ref{LLL Symmetric} implies that there exists a permutation of $[n]$ with no block twins of length $k$. Again, the Stirling formula yields that $k = (1+o(1))\frac{\log n}{\log \log n}$.
\end{proof}

\subsection{Block twins in random permutations} In the previous subsection we used a random permutation $\Pi$ as a tool of the probabilistic method to estimate $bt(n)$.
Now we are interested in block twins in random permutations. 
The result below shows that  the maximum length of block twins in $\Pi$ is a.a.s. roughly twice as big as in the worst case.

\begin{thm}\label{Theorem Block Twins Random}
For a random $n$-permutation $\Pi$, a.a.s.~we have
\[
bt(\Pi) = (2+o(1)) \frac{\log n}{\log \log n}.
\]
\end{thm}

\begin{proof} For $1\le i\le j-k\le n-2k$, recall from the proof of Theorem \ref{Theorem Block Twins} that $\cE_{i,j}$ denotes the event that
 $(\Pi(i),\dots,\Pi(i+k-1))$ and $(\Pi(j),\dots,\Pi(j+k-1))$ are block twins. Let $X_{i,j}$ be the indicator random variable of the event $\cE_{i,j}$, that is $X_{i,j}=1$ if $\cE_{ij}$ holds and $X_{i,j}=0$ otherwise,  and set $X=\sum X_{i,j}$.

By \eqref{1|k!}, we have $\Prob(X_{i,j}=1)=\PP(\cE_{i,j})=1/k!$ and thus $\E(X) = \Theta(n^2/k!)$. Set $\omega(n)$ for any sequence of integers such that $\omega(n)\to\infty$ but $\omega(n)=o(\log\log n/\log\log\log n)$. It is easy to check, using the Stirling formula, that
\[
\log\left(\frac{n^2}{k!}\right) \to
\begin{cases}
-\infty &\text{ if } k = \left\lceil 2 \frac{\log n}{\log\log n}\left(1+\frac1{\omega(n)}\right)\right\rceil\\
+\infty &\text{ if } k = \left\lfloor2\frac{\log n}{\log\log n}\left(1-\frac1{\omega(n)}\right)\right\rfloor.
\end{cases}
\]
Thus, in the former case $\E X=o(1)$ and, by Markov's inequality, a.a.s. there are no block twins of length $k$ in $\Pi$.

We will use the second moment method to show that in  the latter case, when $\E X\to\infty$,
 a.a.s. there is a pair of block twins of length~$k$ in $\Pi$. Note that, in fact, in this case we have
 \begin{equation}\label{nkeO}
 \frac{n^2}{k!}=e^{\Omega(\log n/\omega(n))}.
 \end{equation}

For $1\le i_1\le j_1-k\le n-2k$ and $1\le i_2\le j_2-k\le n-2k$, let $A = \{i_1,\dots,i_1+k-1\}$, $B = \{j_1,\dots,j_1+k-1\}$, $C = \{i_2,\dots,i_2+k-1\}$ and $D = \{j_2,\dots,j_2+k-1\}$.
If either $A\cap (C\cup D)=\emptyset$ or $B\cap (C\cup D)=\emptyset$,
then, by Fact \ref{ABCD},
\[
\Prob(X_{i_1,j_1}=X_{i_2,j_2}=1) =\frac{1}{(k!)^2},
\]
that is, $X_{i_1,j_1}$ and $X_{i_2,j_2}$ are independent.

The number of the remaining pairs of indicators $X_{i_1,j_1}$ and $X_{i_2,j_2}$ is $O(n^2k^2)$.

Hence, using the trivial bound
\[
\Prob(X_{i_1,j_1}=X_{i_2,j_2}=1) \le \Prob(X_{i_1,j_1}=1) = \frac{1}{k!},
\]
we have
$$Var(X)=\sum_{i_1,j_1}\sum_{i_2,j_2}Cov(X_{i_1,j_1},X_{i_2,j_2})=O\left(\frac{n^2k^2}{k!}\right)$$ and, by Chebyshev's inequality and (\ref{nkeO}),
$$\PP(X=0)\le\frac{Var(X)}{(\E X)^2}=O\left(\frac{k^2k!}{n^2}\right)=o(1).$$

\end{proof}

\subsection{Tight twins}

Recall that a pair of twins $(\sigma_1,\sigma_2)$ in a permutation $\pi$ is called \emph{tight} if their union is a block in $\pi$. Note that unlike  general twins and block twins, tight twins are not `monotone', that is the absence of tight twins of length $k$ does not exclude the presence of length bigger than $k$. Let $tt(\pi)$ denote the maximum length of tight twins in $\pi$:
$$tt(\pi)=\max\{\text{$|\sigma_1|:(\sigma_1,\sigma_2)$ is a pair of tight twins in $\pi$}\},$$
and let $$tt(n)=\min\{tt(\pi):\text{$\pi$ is a permutation of $[n]$}\}.$$

We will prove that $tt(n)\leqslant 12$, which means that for every $n$ there exists a permutation of $[n]$ with no tight twins of length $13$ or longer.

\begin{thm}\label{Theorem Tight Twins}
	For every $n\geqslant 1$ we have $tt(n)\leqslant 12$.
\end{thm}

\begin{proof}
	Let $\Pi$ be a random permutation of $[n]$. We will apply Lemma \ref{LLL} in the following setting. For a fixed segment $R$ of length $2r$, let $\cA_R$ denote the event that a sub-permutation of $\Pi$ occupying $R$ consists of tight twins. We consider only segments of length at least~$26$, so we assume that $r\geqslant 13$. Let $V_r$ denote the collection of all such events $\cA_R$ for all possible segments of length $2r$. By (\ref{1|k!}) and the union bound, for every $\cA_R\in V_r$,
$$\PP(\cA_R)\le\frac{1}{2}\binom{2r}{r}/r!.$$ Hence, we may take $p_r=\frac{1}{2}\binom{2r}{r}/r!$.
	
	By Fact \ref{ABCD}, any event $\cA_R$ depends only on those events $\cA_S$ whose segments $S$ intersect~$R$. Hence, if $S$ is any segment of length $2s$, with $s\geqslant 13$ and $S\neq R$, then we may take $\Delta_{rs}=2r+2s-1$. Furthermore, we take $x_s=(2/3)^s$, $s\geqslant 13$.
	
	We are going to prove that for every $r\geqslant 13$
	\[p_{r}\leq	x_{r}\prod\limits_{s=13}^{n/2}(1-x_{s})^{\Delta _{rs}}.\]
	Since $x_s\leqslant1/2$ for $s\ge 13$, we may use the inequality $1-x_s\geqslant e^{-2x_s}$ and obtain the bound
	\begin{multline*}
	\prod\limits_{s=13}^{n/2}(1-x_{s})^{\Delta _{rs}} \geqslant\prod\limits_{s=13}^{n/2}(1-x_s)^{2(r+s)}
	\geqslant \exp\left( -4\sum_{s=13}^{\infty}x_s(r+s)\right)\\=\exp\left( -4r\sum_{s=13}^{\infty}\left(\frac{2}{3}\right)^s\right) \cdot \exp\left(-4\sum_{s=13}^{\infty}s\left(\frac{2}{3}\right)^s\right)\\=\exp\left(-12 r\cdot \left(\frac{2}{3}\right)^{13}\right)\cdot \exp\left(-180\cdot \left(\frac{2}{3}\right)^{13}\right),
	\end{multline*}
	since $\sum_{s=a}^{\infty}\left(\frac{2}{3}\right)^s=3\left(\frac{2}{3}\right)^a$ and $\sum_{s=a}^{\infty}s\left(\frac{2}{3}\right)^s=(3a+6)\left(\frac{2}{3}\right)^a$.
	Therefore, we will be done by showing that
	\[\frac{1}{2}\binom{2r}{r}\cdot \frac{1}{r!}\leqslant \left(\frac{2}{3}\right)^r\cdot \exp\left(-12 r\cdot \left(\frac{2}{3}\right)^{13}\right)\cdot \exp\left(-180\cdot \left(\frac{2}{3}\right)^{13}\right)\]
	which is equivalent to proving that for $r\ge 13$
	\[
	f(r) := \frac{\left(\frac{2}{3}\right)^r\cdot e^{-12 r\cdot \left(\frac{2}{3}\right)^{13}}\cdot e^{-180\cdot \left(\frac{2}{3}\right)^{13}}}{\frac{1}{2}\binom{2r}{r}\cdot \frac{1}{r!}} \ge 1.
	\]
	To this end, observe that $f(r+1)\ge f(r)$ for $r\ge 13$. Indeed,
	\[
	\frac{f(r+1)}{f(r)} = \frac{(r+1)^2}{6r+3}e^{-12 \cdot \left(\frac{2}{3}\right)^{13}} \ge \frac{r+1}{6} e^{-12 \cdot \left(\frac{2}{3}\right)^{13}}
	\ge
	\frac{7}{3} e^{-12 \cdot \left(\frac{2}{3}\right)^{13}} \ge 1.
	\]
	Thus, checking on a calculator that $f(13)\ge 1$ completes the proof.

\end{proof}

Our next result provides an easy lower bound on $tt(n)$.

\begin{prop}\label{Proposition Tight Twins}
We have $tt(n)\geqslant 2$ for all $n\geqslant 6$.
\end{prop}

We start with the following observation.
\begin{claim}\label{claim:tight}
Let $\pi$ be a permutation of length~6. Assume that for some $i\in \{1,2\}$ either $\pi(i) > \pi(i+1) > \pi(i+2)$ or $\pi(i) < \pi(i+1) < \pi(i+2)$. Then, $\pi$ contains a pair of tight twins.
\end{claim}

\begin{proof}
Without loss of generality assume  that $\pi(1) > \pi(2) > \pi(3)$. We will show that one cannot avoid tight twins of length 2 within $(\pi(1),\dots,\pi(5))$.

Suppose to the contrary. By considering $(\pi(1),\dots,\pi(4))$ we must have  $\pi(3)<\pi(4)$.
Furthermore, by considering $(\pi(1),\dots,\pi(4))$ and $(\pi(2),\dots,\pi(5))$, we infer that, respectively, $\pi(1)<\pi(4)$ and $\pi(4)<\pi(5)$. But now $\pi(2)<\pi(5)$ and $\pi(3)<\pi(4)$ yielding a pair of twins in $(\pi(2),\dots,\pi(5))$, a contradiction.
\end{proof}

\begin{proof}[Proof of Proposition~\ref{Proposition Tight Twins}]
Without loss of generality we may assume that $\pi(1)<\pi(2)$. Then, $\pi(3)>\pi(4)$; otherwise we are done. But now either $\pi(1) < \pi(2) < \pi(3)$ or $\pi(2) > \pi(3) > \pi(4)$ and the statement follows from Claim~\ref{claim:tight}.
\end{proof}


\section{Concluding Remarks}

The major open problem concerning twins in permutations is to determine the asymptotic shape of the function $t(n)$. In view of Theorem~\ref{rg2} and the Bukh and Rudenko result it would be nice to know the following.

\begin{prob}
Is $t(n) \gg n^{3/5}$?
\end{prob}

Another challenging problem concerns the case of tight twins. In Theorem \ref{Theorem Tight Twins} we proved that there exist arbitrarily long permutations avoiding tight twins of length at least $13$. How far is this constant from the optimum? On the one hand, by taking $x_{12}=9/500$ and $x_s=(2/3)^s$ for $s\ge 13$ one can show, by a tedious adaptation of the  proof that $tt(n)\le 11$.
On the other hand, in Proposition \ref{Proposition Tight Twins} we demonstrated that one cannot avoid tight twins of length $2$ in any $n$-permutation for $n\ge6$.

\begin{prob}
Is it true that there exist arbitrarily long permutations without tight twins of length at least $3$?
\end{prob}
Here is an example of a permutation of length $18$ avoiding tight twins of length $3$ or more (but we do not know how to generalize it for larger $n$): $$(14,15,16,3,2,1,10,11,12,5,4,18,8,9,17,7,6,13).$$

Let us conclude the paper with a problem in the spirit of Ulam. Two $n$-permutations $\alpha$ and $\beta$ are called \emph{$k$-similar} if they can be split into $k$ sub-permutations, respectively, $\alpha_1,\dots,\alpha_k$ and $\beta_1,\dots,\beta_k$, so that $\alpha_i$ is similar to $\beta_i$ for all $i=1,2,\dots,k$. Let $U(\alpha,\beta)$ be the least number $k$ such that $\alpha$ and $\beta$ are $k$-similar.

\begin{prob}
What is the average value of $U(\alpha,\beta)$ over all pairs of $n$-permutations?
\end{prob}

\end{document}